\documentclass[11pt,a4paper,oneside]{amsart}%
\usepackage{amsmath}
\usepackage{amssymb}
\usepackage{amsfonts}
\usepackage{graphicx}%
\providecommand{\U}[1]{\protect\rule{.1in}{.1in}}
\newtheorem{theorem}{Theorem}
\theoremstyle{plain}

\newtheorem{corollary}{Corollary}

\newtheorem{lemma}{Lemma}

\newtheorem{remark}{Remark}

\numberwithin{equation}{section}
\begin{document}
\title[Quadratic pencil of difference equations]{Quadratic pencil of difference equations: Jost solutions, spectrum, and
principal vectors}
\author[M. Ad\i var]{Murat Ad\i var}
\address{{\footnotesize Izmir University of Economics, Department of Mathematics,
35330, Bal\c{c}ova, Izmir, Turkey}}
\date{}

\begin{abstract}
In this paper, a quadratic pencil of Schr\"{o}dinger type difference operator
$L_{\lambda}$ is taken under investigation to give a general perspective on
the spectral analysis of non-selfadjoint difference equations of second order.
Introducing Jost-type solutions, structural and quantitative properties of
spectrum of the operator $L_{\lambda}$ are analyzed and hence, a discrete
analog of the theory in Degasperis, \ (\emph{J.Math.Phys}. 11: 551--567, 1970)
and Bairamov et. al, (\emph{Quaest. Math.} 26: 15--30, 2003) is developed. In
addition, several analogies are established between difference and
$q$-difference cases. Finally, the principal vectors of $L_{\lambda}$ are
introduced to lay a groundwork for the spectral expansion.\newline%
\newline

\end{abstract}
\maketitle

\section{Introduction}

\label{intro}Let $L_{\lambda}$ denote the quadratic pencil of difference
operator generated in $\ell^{2}(\mathbb{Z})$ by the difference expression%
\[
L_{\lambda}y_{n}=\Delta\left(  a_{n-1}\Delta y_{n-1}\right)  +\left(
q_{n}+2\lambda p_{n}+\lambda^{2}\right)  y_{n},\,\,\,n\in\mathbb{Z},
\]
where $\Delta$ is forward difference operator, $\lambda$ is spectral
parameter, $\left\{  a_{n}\right\}  _{n\in\mathbb{Z}}$, $\left\{
p_{n}\right\}  _{n\in\mathbb{Z}}$, and $\left\{  q_{n}\right\}  _{n\in
\mathbb{Z}}$ are complex sequences satisfying
\begin{equation}%
{\textstyle\sum\nolimits_{n\in\mathbb{Z}}}
\left\vert n\right\vert \left\{  \left\vert 1-a_{n}\right\vert +\left\vert
p_{n}\right\vert +\left\vert q_{n}\right\vert \right\}  <\infty,\label{p-q}%
\end{equation}
and $a_{n}\neq0$ for all $n\in\mathbb{Z}$.

Evidently, Schr\"{o}dinger type difference equation%
\begin{equation}
\Delta(a_{n-1}\Delta y_{n-1})+(q_{n}-\lambda)y_{n}=0\,\,\,\,n\in
\mathbb{Z}\label{SCH}%
\end{equation}
and the difference equation
\begin{equation}
\Delta(a_{n-1}\Delta y_{n-1})+\left(  q_{n}-\lambda\right)  ^{2}%
y_{n}=0,\,\,\,\,\,n\in\mathbb{Z}\label{K-G}%
\end{equation}
of Klein--Gordon type are special cases of the equation%
\begin{equation}
\Delta\left(  a_{n-1}\Delta y_{n-1}\right)  +\left(  q_{n}+2\lambda
p_{n}+\lambda^{2}\right)  y_{n}=0,\,\,\,n\in\mathbb{Z}.\label{1.7.0}%
\end{equation}
Observe that the dependence on spectral parameter $\lambda$ in (\ref{K-G}) and
(\ref{1.7.0}) is non-linear while it is linear in (\ref{SCH}). Also, since the
sequences $\left\{  a_{n}\right\}  _{n\in\mathbb{Z}}$, $\left\{
p_{n}\right\}  _{n\in\mathbb{Z}}$, and $\left\{  q_{n}\right\}  _{n\in
\mathbb{Z}}$ are allowed to take complex values, Eq. (\ref{1.7.0}) is non-selfadjoint.

Note that, the equation (\ref{SCH}) can be rewritten as%
\begin{equation}
a_{n}y_{n+1}+b_{n}y_{n}+a_{n-1}y_{n}=\lambda y_{n},\label{SCH1}%
\end{equation}
where%
\[
b_{n}=q_{n}-a_{n}-a_{n}.
\]
In \cite{guseinov}, Guseinov studied the inverse problem of scattering theory
for the Eq. (\ref{SCH1}), where $\left\{  a_{n}\right\}  _{n\in\mathbb{N}}%
\ $and $\left\{  b_{n}\right\}  _{n\in\mathbb{N}}$ are real sequences
satisfying $a_{n}>0$ and%
\[%
{\textstyle\sum\nolimits_{n=1}^{\infty}}
\left\vert n\right\vert \left(  \left\vert 1-a_{n}\right\vert +\left\vert
b_{n}\right\vert \right)  <\infty.
\]
In \cite{adivar1} and \cite{adivar 2}, the authors investigated spectral
properties of the difference operator associated with Eq. (\ref{SCH1}) in the
case when $\left\{  a_{n}\right\}  _{n\in\mathbb{Z}}\ $and $\left\{
b_{n}\right\}  _{n\in\mathbb{Z}}$ are complex sequences satisfying%
\begin{equation}%
{\textstyle\sum\nolimits_{n\in\mathbb{Z}}}
\left\vert n\right\vert \left(  \left\vert 1-a_{n}\right\vert +\left\vert
b_{n}\right\vert \right)  <\infty.\label{C1}%
\end{equation}
To the best of our knowledge, quantitative properties of spectrum of the
non-selfadjoint difference operators corresponding to Eq. (\ref{K-G}) and Eq.
(\ref{1.7.0}) have not been treated elsewhere before.

In recent years, quantum calculus and $q$-difference equations has taken a
prominent attention in the literature including \cite{adivar3}, \cite{adivar4}%
, \cite{dobrev}, \cite{sergeev}. In particular, \cite{adivar3} and
\cite{adivar4} are concerned with the spectral analysis of $q$-difference
equation
\begin{equation}
\left(  a\left(  t\right)  u^{\Delta}\left(  t\right)  \right)  ^{\Delta\rho
}+\left(  b\left(  t\right)  -\lambda\right)  u\left(  t\right)  =0\text{,
}t=q^{n}\text{ and }n\in\mathbb{Z}.\label{S-q}%
\end{equation}
However, there is lack of literature on the spectral analysis of quadratic
pencil of $q$-difference equation
\begin{equation}
\left(  a\left(  t\right)  u^{\Delta}\left(  t\right)  \right)  ^{\Delta\rho
}+\left(  b\left(  t\right)  +2\mu c\left(  t\right)  +\mu^{2}\right)
u\left(  t\right)  =0\text{, }t=q^{n}\text{ and }n\in\mathbb{Z}\label{Lq}%
\end{equation}
which includes Eq. (\ref{S-q}) as a particular case.

This paper aims to investigate quantitative properties of spectrum of
quadratic pencil difference operator $L_{\lambda}$. This will provide a wide
perspective on spectral analysis of second order difference equations
(\ref{SCH}) and (\ref{K-G}) and avoid deriving results separately. The
remainder of the manuscript is organized as follows: In Section 2, we proceed
by the procedure, which has been developed by Naimark, Lyance, and others,
consisting of the following steps:

\begin{itemize}
\item Formulation of Jost solutions,

\item Determination of the resolvent operator,

\item Description of the sets of eigenvalues and spectral singularities in
terms of singular points of the kernel of the resolvent,

\item Use of boundary uniqueness theorems of analytic functions to provide
sufficient conditions guaranteeing finiteness of eigenvalues and spectral singularities.
\end{itemize}

\noindent Section 3 is concerned with applications of acquired results in
Section 2. This section also contains a brief subsection to show how the
obtained results might be extended to quadratic pencil of $q$-difference
equation (\ref{Lq}). The latter section introduces principal functions of the
operator $L_{\lambda}$.

Therefore, we improve and generalize the results given in \cite{adivar1,
adivar 2, adivar3, adivar4}.

\section{Spectrum}

Hereafter, we assume (\ref{p-q}) unless otherwise stated.

\subsection{\textbf{Jost solutions of Eq. }(\ref{1.7.0})}

The structure of Jost solutions plays a substantial role in spectral analysis
of difference and differential operators. By the next theorem and several
lemmas in this section, we provide an extensive information about the
structure of Jost solutions of Eq. (\ref{1.7.0}).

To show the structural differences between Jost solutions in continuous and
discrete cases, we first consider the Jost solutions of the differential
equations%
\begin{equation}
-y^{\prime\prime}+\left[  q(x)+2\lambda p(x)-\lambda^{2}\right]
y=0\,,\,\,\,\,x\in\mathbb{R}_{+}\label{c-q}%
\end{equation}
and%
\begin{equation}
-y^{\prime\prime}+\left[  q(x)-\lambda\right]  y=0\,,\,\,\,\,x\in
\mathbb{R}_{+}.\label{sturm}%
\end{equation}

While a Jost solution of the quadratic pencil of differential equation
(\ref{c-q}) is given by%

\begin{equation}
e(x,\lambda)=e^{iw(x)+i\lambda x}+\int_{x}^{\infty}A(x,t)e^{i\lambda
t}dt,\ \ \operatorname{Im}\lambda\geq0,\label{e}%
\end{equation}
where $w(x)=\overset{\infty}{\underset{x}{\int}}p(t)dt$ (see \cite{10}), the
Jost solution of Eq. (\ref{sturm}) is obtained as%
\begin{equation}
f(x,\lambda)=e^{i\sqrt{\lambda}x}+\int_{x}^{\infty}B(x,t)e^{i\sqrt{\lambda}%
t}dt,\ \ \operatorname{Im}\sqrt{\lambda}\geq0\label{f}%
\end{equation}
\cite{30}. Note that, $w(x)$ does not appear in (\ref{f}) since $p(x)=0$ in
(\ref{sturm}). The term $w(x)$ in (\ref{e}) makes the spectral analysis of
(\ref{c-q}) quite challenging. For one thing, the set of eigenvalues of Eq.
(\ref{sturm}) lies only in $\mathbb{C}_{+}$ (see \cite{1}) while that of Eq.
(\ref{c-q}) resides both in $\mathbb{C}_{+}$ and $\mathbb{C}_{-}$ (see
\cite{3} and \cite{guseinov2}), where $\mathbb{C}_{+}$ and $\mathbb{C}_{-}$
indicates the open upper and lower half-planes, respectively.

One of the main achievements of this paper is to introduce Jost solutions of
Eq. (\ref{1.7.0}) in a simple structure and to show that there is no such a
difficulty in discrete case. This will enable us to investigate the spectral
analysis of (\ref{1.7.0})\ as it is done for (\ref{SCH}) in which dependence
on $\lambda$ is linear. In discrete case, the Jost solutions of Eq.
(\ref{SCH}) takes the form%
\begin{equation}
e_{n}^{\pm}(z)=\beta_{n}^{\pm}e^{\pm inz}+\sum_{m\in\mathbb{Z}^{\pm}}^{\infty
}A_{n,m}^{\pm}e^{\pm imz},\ \ n\in\mathbb{Z}\label{e-n+}%
\end{equation}
where $\lambda=2\cos z$, $\operatorname{Im}z\geq0$, and $\mathbb{Z}^{\pm}$
denotes the sets of positive and negative integers, respectively (see
\cite{adivar1} and \cite{guseinov}).

Despite the fact that the dependence on spectral parameter $\lambda$ is
non-linear in Eq. (\ref{1.7.0}), the next theorem offers Jost solutions of the
form%
\begin{equation}
f_{n}^{+}\left(  z\right)  =\alpha_{n}^{+}e^{inz}\left(  1+%
{\textstyle\sum\nolimits_{m=1}^{\infty}}
K_{n,m}^{+}e^{imz/2}\right)  ,\ \ \ n\in\mathbb{Z}\label{F1}%
\end{equation}
and
\begin{equation}
f_{n}^{-}\left(  z\right)  =\alpha_{n}^{-}e^{-inz}\left(  1+%
{\textstyle\sum\nolimits_{-\infty}^{m=-1}}
K_{n,m}^{-}e^{-imz/2}\right)  ,\ \ \ n\in\mathbb{Z}\label{F2}%
\end{equation}
which have similar structure to (\ref{e-n+}), i.e., there is no additional
function $\omega$ of $n$ in the exponent of first terms.

\begin{theorem}
\label{thm1}For $\lambda=2\cos\left(  z/2\right)  $ and $z\in\mathbb{C}%
_{+}:=\left\{  z\in\mathbb{C}:\operatorname{Im}z>0\right\}  ,$ (\ref{F1}) and
(\ref{F2}) solve Eq. (\ref{1.7.0}). The coefficients $\alpha_{n}^{\pm}$ and
the kernels $K_{n,m}^{\pm}$ are uniquely expressed in terms of $\left\{
a_{n}\right\}  _{n\in\mathbb{Z}}$, $\left\{  p_{n}\right\}  _{n\in\mathbb{Z}}%
$, and $\left\{  h_{n}\right\}  _{n\in\mathbb{Z}}$ (where $h_{n}%
=2-a_{n}-a_{n-1}+q_{n}$) as follows:.
\begin{align*}
\alpha_{n}^{+}  &  =\left(
{\textstyle\prod\nolimits_{r=n}^{\infty}}
-a_{r}\right)  ^{-1},\\
K_{n,1}^{+}  &  =2%
{\textstyle\sum\nolimits_{r=n+1}^{\infty}}
p_{r},\\
K_{n,2}^{+}  &  =%
{\textstyle\sum\nolimits_{r=n+1}^{\infty}}
\left(  h_{r}+2p_{r}K_{r,1}^{+}\right)  ,\\
K_{n,3}^{+}  &  =%
{\textstyle\sum\nolimits_{r=n+1}^{\infty}}
h_{r}K_{r,1}^{+}+2p_{r}\left(  K_{r,2}^{+}+1\right)  ,\\
K_{n,4}^{+}  &  =%
{\textstyle\sum\nolimits_{r=n+1}^{\infty}}
\left(  1-a_{r}^{2}\right)  +h_{r}K_{r,2}^{+}+2p_{r}\left(  K_{r,3}%
^{+}+K_{r,1}^{+}\right)  ,
\end{align*}%
\[
K_{n,m+4}^{+}=K_{n,m}^{+}+%
{\textstyle\sum\limits_{r=n+1}^{\infty}}
\left(  1-a_{r}^{2}\right)  K_{r+1,m}^{+}h_{r}K_{r,m+2}^{+}+2p_{r}\left(
K_{r,m+1}^{+}+K_{r,m+3}^{+}\right)  ,
\]
for $m=1,2,...;\,\,\,n\in\mathbb{Z}$, and
\begin{align*}
\alpha_{n}^{-}  &  =\left(
{\textstyle\prod\nolimits_{-\infty}^{r=n-1}}
-a_{r}\right)  ^{-1},\\
K_{n,-1}^{-}  &  =2%
{\textstyle\sum\nolimits_{-\infty}^{r=n-1}}
p_{r},\\
\,K_{n,-2}^{-}  &  =%
{\textstyle\sum\nolimits_{-\infty}^{r=n-1}}
\left(  h_{r}+2p_{r}K_{r,-1}^{-}\right)  ,\\
K_{n,-3}^{-}  &  =%
{\textstyle\sum\nolimits_{-\infty}^{r=n-1}}
h_{r}K_{r,-1}^{-}+2p_{r}\left(  K_{r,-2}^{-}+1\right)  ,\\
K_{n,-4}^{-}  &  =%
{\textstyle\sum\nolimits_{-\infty}^{r=n-1}}
\left(  1-a_{r-1}^{2}\right)  +h_{r}K_{r,-2}^{-}+2p_{r}\left(  K_{r,-3}%
^{-}+K_{r,-1}^{-}\right)  ,
\end{align*}%
\[
K_{n,m-4}^{-}=K_{n,m}^{-}+%
{\textstyle\sum\limits_{-\infty}^{r=n-1}}
\left(  1-a_{r-1}^{2}\right)  K_{r-1,m}^{-}+h_{r}K_{r,m-2}^{-}+2p_{r}\left(
K_{r,m-1}^{-}+K_{r,m-3}^{-}\right)  ,
\]
for $m=-1,-2,...;\,\,\,n\in\mathbb{Z}$.
\end{theorem}

\begin{proof}
Using $\lambda=2\cos\left(  z/2\right)  ,$ Eq. (\ref{1.7.0}) can be rewritten
as
\[
a_{n}y_{n+1}+a_{n-1}y_{n-1}+\left(  h_{n}+2\left(  e^{iz/2}+e^{-iz/2}\right)
p_{n}+e^{iz}+e^{-iz}\right)  y_{n}=0.
\]
Substituting $f_{n}^{+}\left(  z\right)  $ for $y_{n}$ in this equation and
comparing the coefficients of $e^{i(n-1)z}$, $e^{i(n-\frac{1}{2})z}$,
$e^{inz}$, $e^{i(n+\frac{1}{2})z}$, and $e^{i(n+1)z}$, we find
\begin{align*}
a_{n-1}\alpha_{n-1}^{+}+\alpha_{n}^{+} &  =0,\\
K_{n,1}^{+}-K_{n-1,1}^{+}+2p_{n} &  =0,\\
\,K_{n,2}^{+}-K_{n-1,2}^{+}+2p_{n}K_{n,1}^{+}+h_{n} &  =0,\\
K_{n,3}^{+}-K_{n-1,3}^{+}+2p_{n}\left(  K_{n,2}^{+}+1\right)  +h_{n}%
K_{n,1}^{+} &  =0,\\
K_{n,4}^{+}-K_{n-1,4}^{+}+\left(  1-a_{n}^{2}\right)  +h_{n}K_{n,2}^{+}%
+2p_{n}\left(  K_{n,1}^{+}+K_{n,3}^{+}\right)   &  =0,
\end{align*}
respectively. From remaining terms we have the following recurrence relation
\begin{align*}
K_{n-1,m+4}^{+}-K_{n,m+4}^{+} &  =K_{n,m}^{+}-a_{n}^{2}K_{n+1,m}^{+}%
+h_{n}K_{n,m+2}^{+}\\
&  +2p_{n}\left(  K_{n,m+1}^{+}+K_{n,m+3}^{+}\right)
\end{align*}
for $m=1,2,...$ and $n\in\mathbb{Z}$. Similarly we obtain
\begin{align*}
a_{n-1}\alpha_{n+1}^{-}+\alpha_{n}^{-} &  =0,\\
K_{n,-1}^{-}-K_{n+1,-1}^{-}+2p_{n} &  =0,\\
K_{n,-2}^{-}-K_{n+1,-2}^{-}+2p_{n}K_{n,-1}^{-}+h_{n} &  =0,\\
K_{n,-3}^{-}-K_{n+1,-3}^{-}+2p_{n}\left(  K_{n,-2}^{-}+1\right)
+h_{n}K_{n,-1}^{-} &  =0,\\
K_{n,-4}^{-}-K_{n+1,4}^{-}+\left(  1-a_{n}^{2}\right)  +h_{n}K_{n,-2}%
^{-}+2p_{n}\left(  K_{n,-1}^{-}+K_{n,-3}^{-}\right)   &  =0,
\end{align*}
from the coefficients of $e^{-i(n+1)z}$, $e^{-i(n+\frac{1}{2})z}$, $e^{-inz}$,
$e^{-i(n-\frac{1}{2})z}$, and $e^{-i(n-1)z}$, respectively. Use of remaining
terms yields
\begin{align*}
K_{n+1,m-4}^{-}-K_{n,m-4}^{-} &  =K_{n,m}^{-}-a_{n}^{2}K_{n-1,m}^{-}%
+h_{n}K_{n,m-2}^{-}\\
&  +2p_{n}\left(  K_{n,m-1}^{-}+K_{n,m-3}^{-}\right)
\end{align*}
for $m=-1,-2,...$ and $n\in\mathbb{Z}$. Above difference equations give the
desired result, whereas convergence of the coefficients $\alpha_{n}^{\pm}$ and
the kernels $K_{n,m}^{\pm}$ is immediate from the condition (\ref{p-q}).
\end{proof}

Different than the solutions (\ref{e-n+}) of (\ref{SCH}) the coefficients
$K_{n,m}^{\pm}$ in (\ref{F1}-\ref{F2}) are determined by recurrence relations
depending on first four terms $K_{n,m}^{\pm}$, $n\in\mathbb{Z}$, $m=\pm
1,\pm2,\pm3,\pm4$ (the ones for $A_{n,m}^{\pm}$ depend on $A_{n,m}^{\pm}$,
$n\in\mathbb{Z}$, $m=\pm1,\pm2$). In addition to these, the exponential
function in the sum contains the half of the complex variable $z$ because of
the transformation $\lambda=2\cos\left(  z/2\right)  $.

In the following lemma, we list some properties of Jost solutions $f^{\pm
}\left(  z\right)  =\left\{  f_{n}^{\pm}\left(  z\right)  \right\}  :$

\begin{lemma}
\label{Jost 1}

\begin{enumerate}
\item[i.] The kernels $K_{n,m}^{\pm}$ satisfy
\begin{align*}
\left\vert K_{n,m}^{+}\right\vert  &  \leq c%
{\textstyle\sum\nolimits_{r=n+\left[  m/2\right]  }^{\infty}}
\left(  \left\vert 1-a_{r}\right\vert +\left\vert p_{r}\right\vert +\left\vert
q_{r}\right\vert \right)  ,\\
\left\vert K_{n,m}^{-}\right\vert  &  \leq C%
{\textstyle\sum\nolimits_{-\infty}^{r=n+ \left[  m/2\right]  +1}}
\left(  \left\vert 1-a_{r}\right\vert +\left\vert p_{r}\right\vert +\left\vert
q_{r}\right\vert \right)  ,
\end{align*}
where $\left[  m/2\right]  $ is the integer part of $m/2,$ and $c$ and $C$ are
positive constants,

\item[ii.] $f_{n}^{\pm}\left(  z\right)  $ are analytic with respect to $z$ in
$\mathbb{C}_{+}:=\left\{  z\in\mathbb{C}:\operatorname{Im}z>0\right\}  $,
continuous in $\overline{\mathbb{C}}_{+}$,

\item[iii.] For$\,\,\,z\in\overline{\mathbb{C}}_{+}$, $f_{n}^{\pm}\left(
z\right)  $ satisfy the following estimates
\begin{align}
f_{n}^{+}\left(  z\right)   &  =\exp\left(  inz\right)  \left[  1+o\left(
1\right)  \right]  \text{ as }n\rightarrow\infty,\label{1.1}\\
f_{n}^{-}\left(  z\right)   &  =\exp\left(  -inz\right)  \left[  1+o\left(
1\right)  \right]  \text{ as }n\rightarrow-\infty,\label{1.1.0}%
\end{align}
and%
\begin{equation}
f_{n}^{\pm}\left(  z\right)  =\alpha_{n}^{\pm}\exp\left(  \pm inz\right)
\left[  1+o\left(  1\right)  \right]  \,\,\text{as }\operatorname{Im}%
z\rightarrow\infty\,\text{\ for }n\in\mathbb{Z}\text{.}\label{1.1.0.1}%
\end{equation}

\end{enumerate}
\end{lemma}

\begin{proof}
The proof is obtained as an implication of (\ref{p-q}), (\ref{F1}), and
(\ref{F2}).
\end{proof}

Let $g^{\pm}\left(  z\right)  =\left\{  g_{n}^{\pm}\left(  z\right)  \right\}
$ denote the solutions of Eq. (\ref{1.7.0}) subject to the conditions
\[
\lim\nolimits_{n\rightarrow\pm\infty}g_{n}^{\pm}\left(  z\right)  e^{inz}=1,
\]
respectively. By making use of Theorem \ref{thm1}, (\ref{1.1}) and
(\ref{1.1.0}) we have the next result:

\begin{lemma}
\label{Jost 2}

\begin{enumerate}
\item[i.] For $z\in\mathbb{C}_{-}:=\left\{  z\in\mathbb{C}:\operatorname{Im}%
z<0\right\}  $%
\[
g^{\pm}\left(  z\right)  =\left\{  f_{n}^{\pm}\left(  -z\right)  \right\}
_{n\in\mathbb{Z}}%
\]
holds,

\item[ii.] $g_{n}^{\pm}\left(  z\right)  $ are analytic with respect to $z$ in
$\mathbb{C}_{-}$, and continuous in $\overline{\mathbb{C}}_{-}$,

\item[iii.] For $\zeta\in\mathbb{R}$,
\[
W\left[  f^{\pm}\left(  \zeta\right)  ,g^{\pm}\left(  \zeta\right)  \right]
=\mp2i\sin\zeta
\]
holds, where the Wronskian of two solutions $u$ and $v$ of Eq. (\ref{1.7.0})
is defined by
\[
W\left[  u,v\right]  =a_{n}\left(  u_{n}v_{n+1}-u_{n+1}v_{n}\right)  ,
\]

\item[iv.] For $\zeta\in\mathbb{R}\backslash\left\{  n\pi:n\in\mathbb{Z}%
\right\}  $ and $\lambda=2\cos\left(  \zeta/2\right)  $%
\begin{equation}
f_{n}^{+}\left(  \zeta\right)  =\psi\left(  \zeta\right)  f_{n}^{-}\left(
\zeta\right)  +\mu\left(  \zeta\right)  g_{n}^{-}\left(  \zeta\right)
,\label{1.4.0}%
\end{equation}
where%
\[
\psi\left(  \zeta\right)  =\frac{W\left[  f^{+}\left(  \zeta\right)
,g^{-}\left(  \zeta\right)  \right]  }{2i\sin\zeta}\text{, }\mu\left(
\zeta\right)  =-\frac{W\left[  f^{+}\left(  \zeta\right)  ,f^{-}\left(
\zeta\right)  \right]  }{2i\sin\zeta}.
\]

\end{enumerate}
\end{lemma}

It is worth noting that the function $\mu$ has an analytic continuation to the
open upper half-plane $\mathbb{C}_{+}$.

\subsection{\textbf{Resolvent and Discrete Spectrum}}

The set of values $\lambda\in\mathbb{C}$ such that $R_{\lambda}\left(
L_{\lambda}\right)  =L_{\lambda}^{-1}$ exists as a bounded operator on
$\ell^{2}(\mathbb{Z})$ is said to be the resolvent set $\rho\left(
L_{\lambda}\right)  $ of $L_{\lambda}$.

Similar to the one in \cite{adivar1}, we formulate the resolvent set
$\rho\left(  L_{\lambda}\right)  $ and the resolvent operator $R_{\lambda
}\left(  L_{\lambda}\right)  $ as follows:%
\begin{equation}
\rho\left(  L_{\lambda}\right)  =\left\{  \lambda=2\cos\frac{z}{2}%
:z\in\mathbb{C}_{+}\text{ and }\Phi\left(  z\right)  \neq0\right\}
\label{resolvent set}%
\end{equation}
and%
\[
R_{\lambda}\left(  L_{\lambda}\right)  \phi_{n}=\sum_{m\in\mathbb{Z}%
}\mathcal{G}_{n,m}\left(  z\right)  \phi_{n}%
\]
for $\lambda=2\cos\frac{z}{2}\in\rho\left(  L_{\lambda}\right)  $ and
$\phi_{n}\in\ell^{2}\left(  \mathbb{Z}\right)  $, $n\in\mathbb{Z}$, where%
\begin{equation}
\mathcal{G}_{n,m}\left(  z\right)  =\left\{
\begin{array}
[c]{c}%
\frac{f_{m}^{-}\left(  z\right)  f_{n}^{+}\left(  z\right)  }{\Phi\left(
z\right)  },\,\,\,m=n-1,n-2,...\\
\frac{f_{m}^{+}\left(  z\right)  f_{n}^{-}\left(  z\right)  }{\Phi\left(
z\right)  },\,\,\,m=n,n+1,...
\end{array}
\right.  ,\label{1.6}%
\end{equation}
and%
\begin{equation}
\Phi\left(  z\right)  :=2i\sin z\mu\left(  z\right)  =W\left[  f^{-}\left(
z\right)  ,f^{+}\left(  z\right)  \right]  .\label{1.5}%
\end{equation}

Notice that the function $\Phi$ is $4\pi$ periodic, analytic in $\mathbb{C}%
_{+}$, and continuous in $\overline{\mathbb{C}}_{+}$. The zeros of the
function $\Phi$ play a substantial role in the formulation of the sets of
eigenvalues and spectral singularities of $L_{\lambda}$. From (\ref{1.4.0}%
-\ref{1.5}) and definition of eigenvalues we arrive at the following conclusion.

\begin{lemma}
[Eigenvalues]\label{lem2} Let $\sigma_{d}\left(  L_{\lambda}\right)  $ denote
the set of eigenvalues of the operator $L_{\lambda}$. Then we have
\[
\sigma_{d}\left(  L_{\lambda}\right)  =\left\{  \lambda=2\cos\frac{z}{2}:z\in
P^{+}\text{ and }\Phi\left(  z\right)  =0\right\}  ,
\]
where
\[
P^{+}=\left\{  z=\eta+i\varphi:\eta\in\left[  -\pi,3\pi\right]  \text{ and
}\varphi>0\right\}  .
\]

\end{lemma}

Note that $L_{\lambda}$ has no eigenvalues on the real line. This is because,
if $\lambda_{0}=2\cos(z_{0}/2)\in\mathbb{R}$ is an eigenvalue, then the
corresponding solution $y_{n}\left(  z_{0}\right)  $ will satisfy the estimate
$y_{n}\left(  z_{0}\right)  =c_{1}e^{inz_{0}}+c_{2}e^{-inz_{0}}+o\left(
1\right)  $ as $n\rightarrow\infty$ contradicting the fact that $y_{n}\in$
$\ell^{2}\left(  \mathbb{Z}\right)  $.

\subsection{\textbf{Continuous spectrum and the spectral singularities}}

To obtain the continuous spectrum of the operator $L_{\lambda}$ we shall
resort to the following lemma.

\begin{lemma}
\label{lem3} For every $\delta>0$, there is a positive number $c_{\delta}$
such that
\[
\left\Vert R_{\lambda}\left(  L_{\lambda}\right)  \right\Vert \geq
\frac{c_{\delta}}{\left\vert \Phi\left(  z\right)  \right\vert \sqrt
{1-\exp\left(  -2\operatorname{Im}z\right)  }}%
\]
for $\lambda=2\cos\frac{z}{2}$, $\,z\in\mathbb{C}_{+}$, and $\operatorname{Im}%
z>\delta$. Hence, $\left\Vert R_{\lambda}\left(  L_{\lambda}\right)
\right\Vert \rightarrow\infty$ $\ $as $\operatorname{Im}z\rightarrow0$.
\end{lemma}

\begin{proof}
Let $\delta>0$, $z\in\mathbb{C}_{+}$ and $\operatorname{Im}z>\delta$. Since
$f_{n}^{+}\left(  z\right)  =e^{inz}+o\left(  1\right)  $ as $n\rightarrow
\infty$ for $\,z\in\mathbb{C}_{+}$, we have
\[
\left\vert f_{n}^{+}\left(  z\right)  \right\vert >\frac{1}{2}%
e^{-n\operatorname{Im}z},
\]
and therefore,
\[
\left\Vert f_{n}^{+}\right\Vert ^{2}\geq\frac{\exp\left(  -2m_{0}%
\operatorname{Im}z\right)  }{4\left(  1-\exp\left(  -2\operatorname{Im}%
z\right)  \right)  }.
\]
Now, we employ the function $h_{m_{0}}$ defined by
\[
h_{m_{0}}\left(  z\right)  :=\left\{
\begin{array}
[c]{c}%
\overline{f_{m}^{-}\left(  z\right)  },\,\,m=m_{0}-1,\,m_{0}-2,...\\
0,\,\,\,\,\,m=m_{0},\,\,m_{0}+1,...
\end{array}
\right.  ,
\]
where $m_{0}\in\mathbb{Z}$ is a negative constant. Evidently $h_{m_{0}}%
(z)\in\ell^{2}\left(  \mathbb{Z}\right)  $ and
\begin{align*}
R_{\lambda}\left(  L_{\lambda}\right)  h_{m_{0}}(z) &  =%
{\textstyle\sum\nolimits_{-\infty}^{m=m_{0}-1}}
\mathcal{G}_{n,m}\left(  z\right)  \overline{f_{m}^{-}}(z)\\
&  =\frac{f_{n}^{+}(z)}{\Phi(z)}\left\Vert h_{m_{0}}\right\Vert ^{2}.
\end{align*}
holds. Thus, we reach the following inequality:
\[
\frac{\left\Vert R_{\lambda}\left(  L_{\lambda}\right)  h_{m_{0}}\right\Vert
^{2}}{\left\Vert h_{m_{0}}\right\Vert ^{2}}\geq\frac{c_{\delta}^{2}}{\left(
1-\exp\left(  -2\operatorname{Im}z\right)  \right)  \left\vert \Phi\left(
z\right)  \right\vert ^{2}}%
\]
as desired, where
\[
c_{\delta}=\frac{\left\Vert h_{m_{0}}\right\Vert }{2\exp\left(  m_{0}%
\delta\right)  }.
\]

\end{proof}

We shall need the following theorem at several occasions in our further work.

\begin{theorem}
\label{thm4} $\sigma_{c}\left(  L_{\lambda}\right)  =\left[  -2,2\right]  $,
where $\sigma_{c}\left(  L_{\lambda}\right)  $ denotes the continuous spectrum
of the operator $L_{\lambda}$.
\end{theorem}

\begin{proof}
By (\ref{resolvent set}), for any $\lambda\in\rho(L_{\lambda})$ there is a
corresponding $z\in\mathbb{C}_{+}$ such that $\lambda=2\cos\frac{z}{2}$ and
$\Phi\left(  z\right)  \neq0$. Let $\lambda_{0}=2\cos\frac{z_{0}}{2}\in
\sigma_{c}\left(  L_{\lambda}\right)  $. Then $\left\Vert R_{\lambda}\left(
L_{\lambda}\right)  \right\Vert \rightarrow\infty$ as $\lambda\rightarrow
\lambda_{0}$. This shows that $\Phi\left(  z\right)  =2i\sin z\mu\left(
z\right)  \rightarrow0$ as $z\rightarrow z_{0}.$ Continuity of $\mu$ and
$\mu\left(  z_{0}\right)  \neq0$ yield $\sin z\rightarrow0$ and
$\operatorname{Im}z\rightarrow0$. On the other hand, we have
$\operatorname{Im}z\rightarrow\operatorname{Im}z_{0}$ since $\lambda
\rightarrow\lambda_{0}$. It follows that $\operatorname{Im}z_{0}=0$, i.e.,
$\lambda_{0}=2\cos\frac{z_{0}}{2}\in\left[  -2,2\right]  $. Conversely, it
follows from Lemma \ref{lem3} that $\left\Vert R_{\lambda}\left(  L_{\lambda
}\right)  \right\Vert \rightarrow\infty$ for $\lambda=2\cos\frac{z}{2}%
\in\left[  -2,2\right]  $. Now, we have to show that the range $\mathcal{R}%
\left(  L_{\lambda}\right)  $ of values of the operator $L_{\lambda}$ is dense
in the space $\ell^{2}\left(  \mathbb{Z}\right)  $. It is obvious that the
orthogonal complement of $\mathcal{R}\left(  L_{\lambda}\right)  $ coincides
with the space of solutions $y\in\ell^{2}\left(  \mathbb{Z}\right)  $ of Eq.
$L_{\lambda}^{\ast}y=0$, where $L_{\lambda}^{\ast}$ denotes the adjoint
operator. Since Eq. $L_{\lambda}^{\ast}y=0$ has no any eigenvalue on the real
line, the orthogonal complement of the set $\mathcal{R}\left(  L_{\lambda
}\right)  $ consists only of the zero element. This completes the proof.
\end{proof}

\begin{remark}
If $p_{n}=0$, the difference equation%
\[
\Delta\left(  a_{n-1}\Delta y_{n-1}\right)  +\left(  q_{n}+2\lambda
p_{n}+\lambda^{2}\right)  y_{n}=0,\,\,\,n\in\mathbb{Z},
\]
turns into%
\begin{equation}
a_{n}y_{n+1}+b_{n}y_{n}+a_{n-1}y_{n-1}=\widetilde{\lambda}y_{n},\,\,\,n\in
\mathbb{Z},\label{eq new}%
\end{equation}
where%
\[
b_{n}=2+q_{n}-a_{n}-a_{n-1}\text{ \ and \ }\widetilde{\lambda}=2-\lambda
^{2}\text{.}%
\]
Namely, in the case $p_{n}=0$, (\ref{p-q}) is equivalent to (\ref{C1}) and
$\widetilde{\lambda}=2-\lambda^{2}$ becomes the new spectral parameter. Thus,
Theorem \ref{thm4} implies that the continuous spectrum of the difference
operator corresponding to (\ref{eq new}) is $\left[  -2,2\right]  $. This
result was obtained in \cite[Theorem 3.1]{adivar1}.
\end{remark}

Spectral singularities are poles of the kernel of the resolvent operator and
are imbedded in the continuous spectrum (\cite[Definition 1.1.]{bai2}).
Analogous to the quadratic pencil of Schr\"{o}dinger operator \cite{3}, from
Theorem \ref{thm4} and (\ref{1.6}) we obtain the set of spectral singularities
of the operator $L_{\lambda}$ as follows:

\begin{corollary}
[Spectral singularities]\label{cor5}Let $\sigma_{ss}\left(  L_{\lambda
}\right)  $ denote the set of spectral singularities. We have%
\[
\sigma_{ss}\left(  L_{\lambda}\right)  =\left\{  \lambda=2\cos\frac{z}{2}:z\in
P_{0}\text{ and }\Phi\left(  z\right)  =0\right\}  ,
\]
where $P_{0}:=\left(  -\pi,3\pi\right)  \backslash\left\{  0,\pi,2\pi\right\}
$.
\end{corollary}

Hereafter, we discuss the quantitative properties of eigenvalues and spectral
singularities. For this purpose, we will make use of boundary uniqueness
theorems of analytic functions \textbf{\cite{7}}. Lemma \ref{lem2} and
Corollary \ref{cor5} show that the problem of investigation of quantitative
properties of eigenvalues and spectral singularities can be reduced to the
investigation of quantitative properties of zeros of the function $\Phi$ in
the semi-strip $P^{+}\cup P_{0}$.

\subsection{\textbf{Quantitative properties of eigenvalues and spectral
singularities}}

Now, we investigate the structure of discrete spectrum and the set of spectral
singularities. In this respect, the sets of eigenvalues and spectral
singularities are analyzed in terms of boundedness, closedness, being
countable, etc.

Theorem \ref{thm1} and (\ref{1.1.0.1}) imply the next result.

\begin{corollary}
\label{cor6}Let $\Phi$ be defined by (\ref{1.5}). Then%
\[
\Phi\left(  z\right)  =\left(
{\textstyle\prod\nolimits_{r\in Z}}
-a_{r}\right)  ^{-1}e^{-iz}\left[  1+o\left(  1\right)  \right]  \text{ \ for
}z\in P^{+}\text{as }\operatorname{Im}z\rightarrow\infty.
\]

\end{corollary}

Let $M_{1}$ and $M_{2}$ denote the sets of zeros of the function $\Phi$ in
$P^{+}$ and $P_{0}$, respectively, i.e.,
\begin{align*}
M_{1} &  =\left\{  z\in P^{+}:\Phi\left(  z\right)  =0\right\}  ,\\
M_{2} &  =\left\{  z\in P_{0}:\Phi\left(  z\right)  =0\right\}  .
\end{align*}
Corollary \ref{cor6} shows boundedness of the set $M_{1}$. Since $\Phi$ is a
$4\pi$ periodic function and is analytic in $P^{+}$, the set $M_{1}$ has at
most countable number of elements. By uniqueness of analytic functions, we
figure out that the limit points of the set $M_{1}$ lie in the closed interval
$\left[  -\pi,3\pi\right]  $. Moreover, we can obtain the closedness and the
property of having zero Lebesgue measure of the set $M_{2}$ as a natural
consequence of boundary uniqueness theorems of analytic functions \cite{7}.
Hence, from Lemma \ref{lem2} and Corollary \ref{cor5} we conclude the following.

\begin{lemma}
\label{lem7} \textbf{\ }The set of eigenvalues $\sigma_{d}\left(  L_{\lambda
}\right)  $ $\,$is bounded and countable, and its accumulation points lie on
the closed interval $\left[  -2,2\right]  $. The set of spectral singularities
$\sigma_{ss}\left(  L_{\lambda}\right)  $ is closed and its Lebesgue measure
is zero.
\end{lemma}

Hereafter, we call the multiplicity of a zero of the function $\Phi$ in
$P^{+}\cup P_{0}$ the multiplicity of corresponding eigenvalue or spectral singularity.

In the next two theorems, we shall employ the following conditions and show
that each guarantees finiteness of eigenvalues, spectral singularities, and
their multiplicities.

\noindent\textbf{Condition 1. }$\sup\nolimits_{n\in Z}\left\{  \exp\left(
\varepsilon\left\vert n\right\vert \right)  (\left\vert 1-a_{n}\right\vert
+\left\vert p_{n}\right\vert +\left\vert q_{n}\right\vert )\right\}  <\infty$
for some $\varepsilon>0.$

\noindent\textbf{Condition 2. }$\sup\nolimits_{n\in Z}\left\{  \exp\left(
\varepsilon\left\vert n\right\vert ^{\delta}\right)  (\left\vert
1-a_{n}\right\vert +\left\vert p_{n}\right\vert +\left\vert q_{n}\right\vert
)\right\}  <\infty$ for some $\varepsilon>0$ and $\frac{1}{2}\leq\delta<1. $

Note that Condition 2 is weaker than Condition 1. If Condition 1 holds, then
we get by (i) of Lemma \ref{Jost 1} that
\[
\left\vert K_{n,m}^{\pm}\right\vert \leq c_{1,2}\exp\left(  \mp\left(
\varepsilon/4\right)  m\right)  \text{ for }n=0,1\text{ and }m=\pm
1,\pm2,...\text{,}%
\]
where $c_{1,2}$ are positive constants. That is, the function $\Phi$ has an
analytic continuation to the lower half-plane $\operatorname{Im}%
z>-\varepsilon/2$. Since $\Phi$ is a $4\pi$ periodic function, this analytic
continuation implies that the bounded sets $M_{1}$ and $M_{2}$ have no any
limit point on the real line. Hence, we have the finiteness of the zeros of
the function $\Phi$ in $P^{+}\cup P_{0}$. These results are complemented by
the next theorem.

\begin{theorem}
\label{thm8} Under Condition 1, the operator $L_{\lambda}$ has finite number
of eigenvalues and spectral singularities, and each of them is of finite multiplicity.
\end{theorem}

\begin{proof}
The proof follows from Lemma \ref{lem7} and the fact that the sets $\sigma
_{d}\left(  L_{\lambda}\right)  $ and $\sigma_{ss}\left(  L_{\lambda}\right)
$ have no limit points.
\end{proof}

Under Condition 2, $\Phi$ has no any analytic continuation, so finiteness of
eigenvalues and spectral singularities cannot be proven in a similar way to
that of Theorem \ref{thm8}.

The following lemma can be proved similar to that of \cite[Lemma 2.2]{adivar
2}:

\begin{lemma}
\label{lem9} If Condition 2 holds, then we have
\[
\left|  \frac{d^{k}\Phi}{d\lambda^{k}}\left(  z\right)  \right|  \leq
A_{k},\,\,\,z\in P^{+},\,\,\,k=0,1,...,
\]
where
\[
A_{k}\leq C4^{k}+Dd^{k}k!k^{k\left(  1/\delta-1\right)  }%
\]
and $C$, $D$, and $d$ are positive constants depending on $\varepsilon$ and
$\delta$.
\end{lemma}

\begin{theorem}
\label{thm10} If Condition 2 holds, then eigenvalues and spectral
singularities of the operator $L_{\lambda}$ are finite, and each of them is of
finite multiplicity.
\end{theorem}

\begin{proof}
Let $M_{3}$ and $M_{4}$ denote the sets of limit points of the sets $M_{1}$
and $M_{2}$, respectively, and $M_{5}$ the set of zeros in $P^{+}$ of the
function $\Phi$ with infinite multiplicity. From the boundary uniqueness
theorem of analytic functions we have the relations
\[
M_{1}\cap M_{5}=\emptyset,\,\,M_{3}\subset M_{2},\,\,M_{5}\subset M_{2},
\]
and using continuity of all derivatives of $\Phi$ on $\left[  -\pi
,3\pi\right]  $ we get that
\[
M_{3}\subset M_{5},\,\,M_{4}\subset M_{5}.
\]
Combining Lemma \ref{lem9} and the uniqueness theorem (see \cite[Theorem
2.3]{adivar 2}), we conclude that
\[
M_{5}=\emptyset.
\]
Thus, countable and bounded sets $M_{1}$ and $M_{2}$ have no limit points. The
proof is complete.
\end{proof}

\section{Applications to special cases}

In this section, the spectral results obtained for Eq. (\ref{1.7.0}) are
applied to the following particular cases:

\begin{enumerate}
\item[Case 1.] $p_{n}=0$

\item[Case 2.] $p_{n}=-v_{n}$ and $q_{n}=v_{n}^{2}$,
\end{enumerate}

in which we obtain the equations (\ref{SCH}) and (\ref{K-G}), respectively.

In addition, we explore some analogies between Eq. (\ref{1.7.0}) and its
$q$-analog (\ref{Lq}) using some transformations and the results obtained in
theorems \ref{thm1}-\ref{thm10}. Finally, we deduce the main results of
\cite[Theorem 5]{adivar3}-\cite{adivar4} in the special case.

In the next we cover the first case.

\subsection{\textbf{Sturm-Liouville type difference equation\label{sc1}}}

It is evident that the substitution $p_{n}=0$ in (\ref{1.7.0}) yields the
Sturm-Liouville type difference equation (\ref{eq new}) whose spectral
parameter is $\widetilde{\lambda}=2-\lambda^{2}$. Denote by $\Lambda$ the
difference operator corresponding to Eq. (\ref{eq new}). In \cite{adivar1},
the authors show that the operator $\Lambda$ has finitely many eigenvalues and
spectral singularities provided that
\begin{equation}
\sup\limits_{n\in\mathbb{Z}}\left\{  \exp\left(  \varepsilon\left\vert
n\right\vert \right)  \left(  \left\vert 1-a_{n}\right\vert +\left\vert
b_{n}\right\vert \right)  \right\}  <\infty\label{1.7}%
\end{equation}
holds for some $\varepsilon>0$. Afterwards, a relaxation of the condition
(\ref{1.7}) is given by \cite[Theorem 2.5]{adivar 2} as follows
\begin{equation}
\sup\limits_{n\in\mathbb{Z}}\left\{  \exp\left(  \varepsilon\left\vert
n\right\vert ^{\delta}\right)  \left(  \left\vert 1-a_{n}\right\vert
+\left\vert b_{n}\right\vert \right)  \right\}  <\infty
,\,\,\,\,\,\,\,\,\,\frac{1}{2}\leq\delta<1.\label{1.8}%
\end{equation}
Note that Conditions 1 and 2 turn into the conditions (\ref{1.7}) and
(\ref{1.8}), respectively. That is, the results \cite[Theorem 4.2]{adivar1}
and \cite[Theorem 2.5]{adivar 2} can be obtained from Theorem \ref{thm8} and
Theorem \ref{thm10} as corollaries.

The second case is handled in the following.

\subsection{\textbf{Klein-Gordon type difference equation\label{sc2}}}

Setting $p_{n}=-v_{n}$ and $q_{n}=v_{n}^{2}$ in Eq. (\ref{1.7.0}), we obtain
the Klein-Gordon type non-selfadjoint difference equation%
\begin{equation}
\Delta(a_{n-1}\Delta y_{n-1})+\left(  v_{n}-\lambda\right)  ^{2}%
y_{n}=0,\,\,\,\,\,n\in\mathbb{Z.}\label{KG1}%
\end{equation}
Observe that Eq. (\ref{KG1}) is more general than the discrete analog of the
differential equation
\begin{equation}
-y^{\prime\prime}+\left(  p\left(  x\right)  -\lambda\right)  ^{2}%
y=0.\label{1.9}%
\end{equation}
Let $\Gamma$ denote the difference operator corresponding to Eq. (\ref{KG1}).
The set of eigenvalues of differential operator corresponding to (\ref{1.9})
was determined by Degasperis \textbf{\cite{6}}, in the case that $p$ is real,
analytic and vanishes rapidly for $x\rightarrow\infty$ (for non-selfadjoint
case see also \cite{bairocky}). However, finiteness of eigenvalues and
spectral singularities of difference operator $\Gamma$ has not been shown
elsewhere before. As a consequence of Theorem \ref{thm8} and Theorem
\ref{thm10}, we derive this result as a corollary.

\begin{corollary}
\label{cor11} If for some $\varepsilon>0$%
\[
\sup\limits_{n\in Z}\left\{  \exp\left(  \varepsilon\left\vert n\right\vert
^{\delta}\right)  \left(  \left\vert 1-a_{n}\right\vert +\left\vert
v_{n}\right\vert \right)  \right\}  <\infty,\,\,\,\,\,\,\,\,\,\frac{1}{2}%
\leq\delta\leq1
\]
holds, then the difference operator $\Gamma$ has finite number of eigenvalues
and spectral singularities, and each of them is of finite multiplicity.
\end{corollary}

\subsection{\textbf{$q$-difference case\label{sc3}}}

We suppose $q>1$ and use the following notations throughout this section:%
\begin{align*}
q^{\mathbb{N}} &  =\left\{  q^{n}:n\in\mathbb{N}\right\}  ,\\
q^{-\mathbb{N}} &  =\left\{  q^{-n}:n\in\mathbb{N}\right\}  ,\\
q^{\mathbb{Z}} &  =\left\{  q^{n}:n\in\mathbb{Z}\right\}  .
\end{align*}
The $q$-difference equation is an equation which contains $q$-derivative of
its unknown function. The $q$-derivative is given by
\[
y^{\Delta}\left(  t\right)  =\frac{y\left(  qt\right)  -y\left(  t\right)
}{\left(  q-1\right)  t},\,\,t\in q^{\mathbb{Z}},
\]
and the $q$-integral is defined by
\[
\int_{a}^{b}f\left(  t\right)  \Delta_{q}t=\left(  q-1\right)  \sum
_{t\in\left[  a,b\right)  \cap q^{\mathbb{Z}}}tf\left(  t\right)  .
\]
We shall denote by $\ell^{2}\left(  q^{\mathbb{Z}}\right)  $ the Hilbert space
of square integrable functions with the norm
\[
\left\Vert f\right\Vert _{q}^{2}=\int_{q^{\mathbb{Z}}}\left\vert f\left(
t\right)  \right\vert ^{2}\Delta_{q}t.
\]
Consider the quadratic pencil of Schr\"{o}dinger type $q$-difference operator
$L_{\lambda}^{q}$ corresponding to the equation
\begin{equation}
\left(  a\left(  t\right)  u^{\Delta}\left(  t\right)  \right)  ^{\Delta\rho
}+\left(  b\left(  t\right)  +2\lambda c\left(  t\right)  +\lambda^{2}\right)
u\left(  t\right)  =0\text{, }t\in q^{\mathbb{Z}}\label{0}%
\end{equation}
where $a$, $b$, and $c$ are complex valued functions, $\lambda$ is spectral
parameter and $\rho$ is the backward jump operator defined by
\[
u^{\rho}\left(  t\right)  =u\left(  t/q\right)  ,\ \ t\in q^{\mathbb{Z}%
}\text{.}%
\]
Multiplying Eq. (\ref{0}) by $\sqrt{t/q}$ we arrive at%
\begin{equation}
\widehat{a}\left(  t\right)  \widehat{u}\left(  qt\right)  +\widehat{a}\left(
t/q\right)  \widehat{u}\left(  t/q\right)  +\left\{  \widehat{b}\left(
t\right)  +2\widehat{\lambda}\widehat{c}\left(  t\right)  +\widehat{\lambda
}^{2}\right\}  \widehat{u}\left(  t\right)  =0,\label{1}%
\end{equation}
where $t\in q^{\mathbb{Z}}$ and%
\begin{equation}
\widehat{u}\left(  t\right)  =\sqrt{t}u\left(  t\right)  ,\ \ \ \widehat
{\lambda}=q^{-1/4}\lambda\label{2}%
\end{equation}%
\begin{align*}
\widehat{a}\left(  t\right)   &  =\frac{a\left(  t\right)  }{\left(
q-1\right)  ^{2}t^{2}},\\
\widehat{b}\left(  t\right)   &  =\frac{b\left(  t\right)  }{\sqrt{q}}%
-\sqrt{q}\widehat{a}\left(  t\right)  -\frac{\widehat{a}\left(  t/q\right)
}{\sqrt{q}},\\
\widehat{c}\left(  t\right)   &  =q^{-1/4}c\left(  t\right)  .
\end{align*}
Using the notations
\begin{equation}
t=q^{n}\text{, }\widehat{a}\left(  q^{n}\right)  =\widehat{a}_{n}%
,\ \ \widehat{b}\left(  q^{n}\right)  =\widehat{b}_{n},~\ \widehat{c}\left(
q^{n}\right)  =\widehat{c}_{n},\ \ \widehat{u}\left(  q^{n}\right)
=\widehat{u}_{n}\label{2.0}%
\end{equation}
we can express Eq. (\ref{1}) in the following form
\begin{equation}
\widehat{a}_{n}\widehat{u}_{n+1}+\widehat{a}_{n-1}\widehat{u}_{n-1}+\left\{
\widehat{b}_{n}+2\widehat{\lambda}\widehat{c}_{n}+\widehat{\lambda}%
^{2}\right\}  \widehat{u}_{n}=0,\ \ n\in\mathbb{Z}.\label{3}%
\end{equation}
By establishing a linkage between equations (\ref{0}) and (\ref{3}), the next
theorem provides an information about some spectral properties of $L_{\lambda
}^{q}$.

\begin{theorem}
\label{qdiff} Let the sequence $\widehat{u}=\left\{  \widehat{u}_{n}\right\}
_{n\in\mathbb{Z}}$ and the value $\widehat{\lambda}$ be defined as in
(\ref{2}) and (\ref{2.0}). The following properties hold:

\begin{enumerate}
\item[i.]  $u\in\ell^{2}\left(  q^{\mathbb{Z}}\right)  $ and solves (\ref{0})
if and only if $\widehat{u}\in\ell^{2}\left(  \mathbb{Z}\right)  $ and solves
(\ref{3})$.$

\item[ii.] $\lambda$ is an eigenvalue of (\ref{0}) if and only if
$\widehat{\lambda}=q^{-1/4}\lambda$ is an eigenvalue of (\ref{3}).

\item[iii.] For $\lambda=2q^{1/4}\cos\left(  z/2\right)  $%
\[
J^{+}\left(  t,z\right)  =\alpha^{+}(t)\exp\left(  i\frac{\ln t}{\ln
q}z\right)  \left\{  1+%
{\displaystyle\int\nolimits_{r\in q^{\mathbb{N}}}}
A^{+}\left(  t,r\right)  \exp\left(  i\frac{\ln r}{2\ln q}z\right)  {\Delta
}_{q}r\right\}
\]
and
\[
J^{-}\left(  t,z\right)  =\alpha^{-}(t)\exp\left(  -i\frac{\ln t}{\ln
q}z\right)  \left\{  1+%
{\displaystyle\int\nolimits_{r\in q^{-\mathbb{N}}}}
A^{-}\left(  t,r\right)  \exp\left(  -i\frac{\ln r}{2\ln q}z\right)  {\Delta
}_{q}r\right\}
\]
are Jost solutions of the Equation (\ref{0}), where $\alpha^{\pm}(t)$ and
$A^{\pm}\left(  t,r\right)  $ can be uniquely expressed in terms of the
sequences $\left\{  \widehat{a}_{n}\right\}  _{n\in\mathbb{Z}},\ \ \left\{
\widehat{b}_{n}\right\}  _{n\in\mathbb{Z}}$, $\left\{  \widehat{c}%
_{n}\right\}  _{n\in\mathbb{Z}}$ provided that the condition
\[
\sum_{t\in q^{\mathbb{Z}}}\left\vert \frac{\ln t}{\ln q}\right\vert \left(
\left\vert 1-\widehat{a}\left(  t\right)  \right\vert +\left\vert
2-\widehat{b}\left(  t\right)  \right\vert +\left\vert \widehat{c}\left(
t\right)  \right\vert \right)  <\infty
\]
holds.

\item[iv.]
\begin{align*}
\sigma_{d}\left(  L_{\lambda}^{q}\right)   &  =\left\{  \lambda=2q^{1/4}%
\cos\left(  \frac{z}{2}\right)  :z\in P^{+}\text{ and }Q\left(  z\right)
=0\right\}  ,\\
\sigma_{ss}\left(  L_{\lambda}^{q}\right)   &  =\left\{  \lambda=2q^{1/4}%
\cos\left(  \frac{z}{2}\right)  :z\in P_{0}\text{ and }Q\left(  z\right)
=0\right\}  ,
\end{align*}
where
\[
Q\left(  z\right)  =W\left[  J^{-}\left(  t,z\right)  ,J^{+}\left(
t,z\right)  \right]  .
\]

\item[v.]
\[
\sigma_{c}\left(  L_{\lambda}^{q}\right)  =\left[  -2q^{1/4},2q^{1/4}\right]
,
\]
where $\sigma_{c}\left(  L_{\lambda}^{q}\right)  $ denotes the continuous
spectrum of the operator $L_{\lambda}^{q}$.

\item[vi.] If the condition
\[
\sup_{t\in q^{\mathbb{Z}}}\left\{  \exp\left(  \varepsilon\left\vert \frac{\ln
t}{\ln q}\right\vert ^{\delta}\right)  \left(  \left\vert 1-\widehat{a}\left(
t\right)  \right\vert +\left\vert 2-\widehat{b}\left(  t\right)  \right\vert
+\left\vert \widehat{c}\left(  t\right)  \right\vert \right)  \right\}
<\infty,\ \ \ \frac{1}{2}\leq\delta\leq1
\]
holds for some $\varepsilon>0$, then the quadratic pencil of $q$-difference
operator has finite number of eigenvalues and spectral singularities with
finite multiplicity.
\end{enumerate}
\end{theorem}

\begin{proof}
The proof can be done as that of related results in preceeding sections. For
brevity we only give the outlines. From (\ref{2}) we have (i) and (ii.). Using
$n=\frac{\ln t}{\ln q}$ and (i), proof of (iii) can be obtained in a similar
way to that of Theorem \ref{thm1}. Hence, Lemma \ref{lem2} along with
Corollary \ref{cor5} implies (iv). Using (\ref{0}), (\ref{2}), (\ref{3}), and
Theorem \ref{thm8} we obtain (v). Finally, combining Theorems \ref{thm8},
\ref{thm10}, and (\ref{3}), we conclude (vi).
\end{proof}

\begin{remark}
Theorem \ref{qdiff} not only covers the results of \cite[Theorem 5]{adivar3}
and \cite[Theorem 4]{adivar4} but also derives spectral properties of
Klein-Gordon type $q$-difference equations in the special case $b(t)=v^{2}(t)$
and $c(t)=-v(t)$.
\end{remark}

\section{\textbf{Principal vectors}}

\label{sec:3}

In this section, we determine the principal vectors of $L_{\lambda}$ and
discuss their convergence properties. Thus, we will have an information about
principal vectors of the operators $\Lambda$, $\Gamma$, and $L_{\lambda}^{q}$.

Define
\begin{align*}
F_{n}^{+}\left(  \lambda\right)   &  :=f_{n}^{+}\left(  2\arccos\frac{\lambda
}{2}\right)  ,\,\,\,n\in\mathbb{Z},\\
F_{n}^{-}\left(  \lambda\right)   &  :=f_{n}^{-}\left(  2\arccos\frac{\lambda
}{2}\right)  ,\,\,\,n\in\mathbb{Z},\\
H\left(  \lambda\right)   &  :=\Phi\left(  2\arccos\frac{\lambda}{2}\right)  .
\end{align*}
Obviously, $F^{\pm}\left(  \lambda\right)  =\left\{  F_{n}^{\pm}\left(
\lambda\right)  \right\}  $ solve Eq. (\ref{1.7.0}), and
\[
H\left(  \lambda\right)  =W\left[  F^{-}\left(  \lambda\right)  ,F^{+}\left(
\lambda\right)  \right]
\]
is satisfied. Furthermore, $F^{\pm}$ and $H$ are analytic in $\Theta
=\mathbb{C}\backslash\left[  -2,2\right]  $ and continuous up to the boundary
of $\Theta$. Using Lemma \ref{lem2} and Corollary \ref{cor5} we can state the
sets $\sigma_{d}\left(  L_{\lambda}\right)  $ and $\sigma_{ss}\left(
L_{\lambda}\right)  $ as the sets of zeros of the function $H$ in $\Theta$ and
in $\left[  -2,2\right]  ,$ respectively. Moreover,
\[
\sup\limits_{n\in Z}\left\{  \exp\left(  \varepsilon\left\vert n\right\vert
^{\delta}\right)  (\left\vert 1-a_{n}\right\vert +\left\vert p_{n}\right\vert
+\left\vert q_{n}\right\vert )\right\}  <\infty,\,\,\,\,\,\,\,\,\,\frac{1}%
{2}\leq\delta\leq1,\varepsilon>0
\]
is the condition that guarantees finiteness of zeros $H$ in $\Theta$ and in
$\left[  -2,2\right]  $. Let $\lambda_{1},...,\lambda_{s}$ denote the zeros of
the functions $H$ in $\Theta$ (which are the eigenvalues of $L_{\lambda}$)
with multiplicities $m_{1},...,m_{s}$, respectively. Similarly, let
$\lambda_{s+1},...,\lambda_{k}$ be zeros of the functions $H$ in $\left[
-2,2\right]  $ (which are the spectral singularities of $L_{\lambda}$) with
multiplicities $m_{s+1},...,m_{k},$ respectively. Similar to that of
\cite[Theorem 5.1]{adivar1} one can prove the next result.

\begin{theorem}
\label{thm15}
\[
\left\{  \frac{d^{r}}{d\lambda^{r}}F_{n}^{+}\left(  \lambda\right)  \right\}
_{\lambda=\lambda_{j}}=\sum_{v=0}^{r}\left(
\begin{array}
[c]{c}%
r\\
v
\end{array}
\right)  \beta_{r-v}^{+}\left\{  \frac{d^{v}}{d\lambda^{v}}F_{n}^{-}\left(
\lambda\right)  \right\}  _{\lambda=\lambda_{j}},n\in\mathbb{Z},
\]
holds for $r=0,1,...,m_{j-1}$, $j=1,2,...,k.$
\end{theorem}

Let us introduce the vectors
\begin{equation}
U^{\left(  r\right)  }\left(  \lambda_{j}\right)  :=\left\{  U_{n}^{\left(
r\right)  }\left(  \lambda_{j}\right)  \right\}  _{n\in\mathbb{Z}},\label{U}%
\end{equation}
for $r=0,1,...,m_{j-1}$, $j=1,2,...,k$, where
\begin{align}
U_{n}^{\left(  r\right)  }\left(  \lambda_{j}\right)   &  =\frac{1}%
{r!}\left\{  \frac{d^{r}}{d\lambda^{r}}F_{n}^{+}\left(  \lambda\right)
\right\}  _{\lambda=\lambda_{j}}\nonumber\\
&  =\sum_{v=0}^{r}\frac{\beta_{r-v}^{+}}{\left(  r-v\right)  !}\frac{1}%
{v!}\left\{  \frac{d^{v}}{d\lambda^{v}}F_{n}^{-}\left(  \lambda\right)
\right\}  _{\lambda=\lambda_{j}},n\in\mathbb{Z}.\label{U-F}%
\end{align}
Define the difference expression $\ell_{\lambda}U^{\left(  n\right)  }$ by
\[
\ell_{\lambda}U^{\left(  n\right)  }\left(  \lambda_{j}\right)  =\Delta\left(
a_{n-1}\Delta U^{\left(  n-1\right)  }\right)  +\left(  q_{n}+2\lambda
_{j}p_{n}+\lambda_{j}^{2}\right)  U^{\left(  n\right)  },\,\,\,n\in\mathbb{Z}.
\]
We get by Eq. (\ref{1.7.0}) that
\begin{align*}
\ell_{\lambda}U^{\left(  0\right)  }\left(  \lambda_{j}\right)   &  =0,\\
\ell_{\lambda}U^{\left(  1\right)  }\left(  \lambda_{j}\right)  +\frac{1}%
{1!}\frac{d\ell_{\lambda}}{d\lambda}U^{\left(  0\right)  }\left(  \lambda
_{j}\right)   &  =0,\\
\ell_{\lambda}U^{\left(  r\right)  }\left(  \lambda_{j}\right)  +\frac{1}%
{1!}\frac{d\ell_{\lambda}}{d\lambda}U^{\left(  r-1\right)  }\left(
\lambda_{j}\right)  +\frac{1}{2!}\frac{d^{2}\ell_{\lambda}}{d\lambda^{2}%
}U^{\left(  r-2\right)  }\left(  \lambda_{j}\right)   &  =0
\end{align*}
for $r=2,3,...,m_{j-1}$, $j=1,2,...,k$. This shows that $U^{\left(  r\right)
}\left(  \lambda_{j}\right)  :=\left\{  U_{n}^{\left(  r\right)  }\left(
\lambda_{j}\right)  \right\}  _{n\in\mathbb{Z}}$ for $r=0,1,...,m_{j-1}$,
$j=1,2,...,s$ are principal vectors corresponding to the eigenvalues
$\lambda_{1},...,\lambda_{s}$ of the operator $L_{\lambda}$. The principal
vectors corresponding to the spectral singularities $\lambda_{s+1}%
,...,\lambda_{k}$ are found similarly.

Let us define the Hilbert spaces
\[
H_{p}\left(  \mathbb{Z}\right)  :=\left\{  y=\left\{  y_{n}\right\}
_{n\in\mathbb{Z}}:\sum_{n\in\mathbb{Z}}\left(  1+\left|  n\right|  \right)
^{2p}\left|  y_{n}\right|  ^{2}<\infty\right\}  ,
\]
\[
H_{-p}\left(  \mathbb{Z}\right)  :=\left\{  y=\left\{  y_{n}\right\}
_{n\in\mathbb{Z}}:\sum_{n\in\mathbb{Z}}\left(  1+\left|  n\right|  \right)
^{-2p}\left|  y_{n}\right|  ^{2}<\infty\right\}
\]
for $p=0,1,...$. Evidently $H_{0}\left(  \mathbb{Z}\right)  =\ell^{2}\left(
\mathbb{Z}\right)  $ and
\[
H_{p+1}\subsetneqq H_{p}\subsetneqq\ell^{2}\left(  \mathbb{Z}\right)
\subsetneqq H_{-p}\subsetneqq H_{-p-1}.
\]

Convergence properties of principal vectors of $L_{\lambda}$ are given in the
following theorem.

\begin{theorem}
\label{thm 16} We have

\begin{enumerate}
\item[i.] $U^{\left(  r\right)  }\left(  \lambda_{j}\right)  :=\left\{
U_{n}^{\left(  r\right)  }\left(  \lambda_{j}\right)  \right\}  _{n\in
\mathbb{Z}}\in\ell^{2}\left(  \mathbb{Z}\right)  $ for\thinspace
\thinspace$r=0,1,...,m_{j-1}$, $j=1,2,...,s,$

\item[ii.] $U^{\left(  r\right)  }\left(  \lambda_{j}\right)  :=\left\{
U_{n}^{\left(  r\right)  }\left(  \lambda_{j}\right)  \right\}  _{n\in
\mathbb{Z}}\notin\ell^{2}\left(  \mathbb{Z}\right)  $ for\thinspace
\thinspace$r=0,1,...,m_{j-1}$, $j=s+1,s+2,...,k,$

\item[iii.] $U^{\left(  r\right)  }\left(  \lambda_{j}\right)  :=\left\{
U_{n}^{\left(  r\right)  }\left(  \lambda_{j}\right)  \right\}  _{n\in
\mathbb{Z}}\in H_{-p_{0}+1}$ for $r=0,1,...,m_{j-1}$, $j=s+1,s+2,...,k$,
where
\[
p_{0}=\max\left\{  m_{1},m_{2},...,m_{s},m_{s+1},...,m_{k}\right\}  .
\]

\end{enumerate}
\end{theorem}

\begin{proof}
Use (\ref{U}-\ref{U-F}) and proceed by a method as in \cite[Theorem
5.2]{adivar1} and \cite[Lemma 5.1]{adivar1}.
\end{proof}

\begin{remark}
Employing the transformations in Subsections \ref{sc1}, \ref{sc2}, and
\ref{sc3} in the equalities (\ref{U}-\ref{U-F}), one may derive the principal
vectors of the operators $\Lambda$, $\Gamma$, and $L_{\lambda}^{q}$ easily.
Moreover, Theorem \ref{thm 16} enables us to see the convergence properties of
principal vectors of the operators $\Lambda$, $\Gamma$, and $L_{\lambda}^{q}$.
\end{remark}

\textbf{Open Problem:} The eigenfunction expansion has not been studied even
for the above mentioned particular cases. So, this may be the topic of further studies.

\begin{center}
\textbf{Acknowledgement}
\end{center}

The author would like to thank to anonymous referee for her/his valuable and
constructive comments that help improve this paper.


\begin{thebibliography}{99}                                                                                               %
\bibitem {adivar1}M. Ad\i var and E. Bairamov, Spectral properties of
non-selfadjoint difference operators, \emph{J. Math. Anal. Appl.} 261 (2001), 461-478.

\bibitem {adivar 2}M. Ad\i var and E. Bairamov, Difference equations of second
order with spectral singularities. \emph{J. Math. Anal. Appl.} 277 (2003), no.
2, 714--721.

\bibitem {adivar3}M. Ad\i var and M. Bohner, Spectral Analysis of
$q$-difference Equations with Spectral Singularities, \emph{Mathematical and
Computer Modelling} 43, 695-703, 2006.

\bibitem {adivar4}M. Ad\i var and M. Bohner, Spectrum and principal vectors of
second order $q$-difference equations. \emph{Indian J. Math.} 48 (2006), no.
1, 17--33.

\bibitem {3}E. Bairamov,
{\"{O}}. {\c{C}}akar%
, and A.M. Krall, Spectral properties, including spectral singularities, of a
quadratic pencil of Schr\"{o}dinger operators on the whole real axis.
\emph{Quaest. Math.} 26 (2003), no. 1, 15--30.

\bibitem {bairocky}E. Bairamov, Spectral properties of the non-homogeneous
Klein-Gordon $s$-wave equations. \emph{Rocky Mountain J. Math.} 34 (2004), no.
1, 1--11.

\bibitem {6}A. Degasperis, On the Inverse Problem for the Klein-Gordon s-wave
Equation, \emph{J.Math.Phys}. 11, (1970), 551-567.

\bibitem {dobrev}V. K. Dobrev; P. Truini, and L. C. Biedenharn, Representation
theory approach to the polynomial solutions of $q$-difference equations: \$U%
$\backslash$%
sb q(\{%
$\backslash$%
rm sl\}(3))\$ and beyond. \emph{J. Math. Phys.} 35 (1994), no. 11, 6058--6075.

\bibitem {7}E. P. Dolzhenko, Boundary Value Uniqueness Theorems for Analytic
Functions, \emph{Math.Notes}$.$ 25, No 6, (1979), 437-442.

\bibitem {greiner}W. Greiner, Relativistic Quantum Mechanics, Wave Equations,
Springer Verlag, 1994.

\bibitem {guseinov}G. Sh. Guseinov, The inverse problem of scattering theory
for a second order difference equation on the whole real line. (Russian) Dokl.
Akad. Nauk SSSR 230 (1976), no. 5, 1045--1048. \{English translation: Soviet
Math. Dokl. 17 (1976), no. 6, 1684-1688 (1977).\}

\bibitem {guseinov2}G. Sh. Guseinov, On the spectral analysis of a quadratic
pencil of Sturm-Liouville operators. (Russian) Dokl. Akad. Nauk SSSR 285
(1985), no. 6, 1292--1296.

\bibitem {10}M. Jaulent and C. Jean, The inverse $s$-wave scattering problem
for a class of potentials depending on energy, \emph{Comm. Math. Phys.} 28
(1972), 177-220.

\bibitem {1}A. M. Krall, E. Bairamov, and
{\"{O}}. {\c{C}}akar%
, Spectrum and spectral singularities of a quadratic pencil of a
Schr\"{o}dinger operator with a general boundary condition. J. Differential
Equations 151 (1999), no. 2, 252--267.

\bibitem {bai2}A. M. Krall, E. Bairamov,
{\"{O}}. {\c{C}}akar%
, Spectral analysis of non-selfadjoint discrete Schr\"{o}dinger operators with
spectral singularities. Math. Nachr. 231 (2001), 89--104.

\bibitem {27}F. G. Maksudov and G. Sh. Guseinov, On solution of the inverse
scattering problem for a Quadratic Pencil of one-dimensional Schr\"{o}dinger
operators on the whole axis, \emph{Sov. Math.Dokl.} 34 (1987), 34-38.

\bibitem {30}M. A. Naimark, Investigation of the spectrum and the expansion in
eigenfunctions of a non-selfadjoint operator of second order on a semi-axis,
AMS Translations, 2(16), (1960), 103-193.

\bibitem {sergeev}S. M. Sergeev, A quantization scheme for modular
$q$-difference equations. (Russian) Teoret. Mat. Fiz. 142 (2005), no. 3,
500--509; translation in \emph{Theoret. and Math. Phys}. 142 (2005), no. 3, 422--430.
\end{thebibliography}
\end{document}